\documentclass{amsart}
\usepackage[utf8]{inputenc}
\usepackage{amsmath}
\usepackage{amsthm}
\usepackage{hyperref}
\usepackage{cleveref}
\usepackage{amssymb}
\usepackage{todonotes}
\usepackage{tikz}
\usepackage{tikz-cd}
\usepackage{mathtools}
\usepackage{todonotes}
\usepackage{enumerate}

\newtheorem{theorem}{Theorem}
\numberwithin{theorem}{section}
\newtheorem{lemma}[theorem]{Lemma}
\newtheorem{proposition}[theorem]{Proposition}
\newtheorem{conjecture}[theorem]{Conjecture}
\newtheorem{corollary}[theorem]{Corollary}

\theoremstyle{definition}
\newtheorem{definition}[theorem]{Definition}
\newtheorem{example}{Example}

\theoremstyle{remark}
\newtheorem{remark}[theorem]{Remark}

\newcommand{\cH}{\mathcal{H}}
\newcommand{\cA}{\mathcal{A}}
\newcommand{\cY}{\mathcal{Y}}
\newcommand{\cI}{\mathcal{I}}
\newcommand{\tR}{\widetilde{R}}
\newcommand{\tH}{\widetilde{H}}
\newcommand{\bs}{\boldsymbol}

\DeclareMathOperator{\dc}{dc}
\DeclareMathOperator{\Des}{Des}
\DeclareMathOperator{\spn}{span}

\usetikzlibrary{decorations.markings, arrows.meta}
\tikzcdset{
 arrow style=tikz,
 diagrams={>={Straight Barb[scale=0.8]}},
 marrow/.style={decoration={markings,mark=at position 0.5 with {\arrow{#1}}}, postaction=decorate},
 rmarrow/.style={decoration={markings,mark=at position 0.8 with {\arrow{#1}}}, postaction=decorate}
}
\begin{document}

\title[Combinatorial invariance for $R$-polynomials of elementary intervals]{Combinatorial invariance for Kazhdan--Lusztig $R$-polynomials of elementary intervals}
\author{Grant T. Barkley}
\thanks{GTB was supported by a National Science Foundation grant DMS-1854512.}
\address[Barkley]{Department of Mathematics, Harvard University, Cambridge, MA.}
\email{{\href{mailto:gbarkley@math.harvard.edu}{gbarkley@math.harvard.edu}}}

\author{Christian Gaetz}
\address[Gaetz]{Department of Mathematics, University of California, Berkeley, CA.}
\email{{\href{mailto:crgaetz@gmail.com}{{\tt gaetz@berkeley.edu}}}}

\date{\today}

\begin{abstract}
We adapt the \emph{hypercube decompositions} introduced by Blundell--Buesing--Davies--Veli\v{c}kovi\'{c}--Williamson to prove the Combinatorial Invariance Conjecture for Kazhdan--Lusztig $R$-polynomials in the case of \emph{elementary intervals} in $S_n$. This significantly generalizes the main previously-known case of the conjecture for $S_n$, that of lower intervals. 
\end{abstract}
\maketitle

\section{Introduction}

Indexed by pairs of elements $u,v$ in the symmetric group $S_n$ (or more generally, any Coxeter group $W$) the \emph{Kazhdan--Lusztig} polynomials $P_{u,v}(q)$ are of fundamental importance in geometric representation theory \cite{Kazhdan-Lusztig-1}. They mediate deep connections between canonical bases of Hecke algebras, representations of Lie algebras, and the geometry of Schubert varieties, and their study has led to the introduction of many of the central techniques in the field \cite{Beilinson-Bernstein,Brylinski-Kashiwara,Elias-Williamson, Kazhdan-Lusztig-2}. This makes Conjecture~\ref{conj:cic}, known as the Combinatorial Invariance Conjecture, all the more remarkable. It says that this rich geometric and representation-theoretic data in fact depends only on the Bruhat interval $[u,v]$ as an abstract poset (equivalently, on the cell adjacencies in the Schubert stratification). 
We call an object indexed by a pair $u \leq v$ \emph{combinatorial} if it only depends on the isomorphism type of the interval $[u,v]$.

\begin{conjecture}[Lusztig 1980s; Dyer \cite{Dyer1987}]
\label{conj:cic}
The Kazhdan--Lusztig polynomial is combinatorial. In other words, for $u,v \in W$ and $u',v' \in W'$, if $[u,v] \cong [u',v']$, then
\[P_{u,v}(q)=P_{u',v'}(q).\]
\end{conjecture}

This conjecture was circulated by Lusztig in the early 1980s and was first suggested in print by Dyer \cite[Rem.~7.31]{Dyer1987}. It has received significant attention. The most general results to date have established the case of \emph{lower intervals} for general $W$, where both $u$ and $u'$ are the identity element $e$ \cite{Brenti-lower-intervals, Lower-intervals-general-type, Delanoy, DuCloux}. Other known cases include intervals in $S_n$ of length at most 8 \cite{Incitti,Incitti2007} and all intervals in the rank-three affine Weyl group of type $\widetilde{A_2}$ \cite{affine-A2}. Patimo \cite{Patimo-q-coefficient} has also shown that the coefficient of $q$ in $P_{u,v}(q)$ is combinatorial for an arbitrary interval in $S_n$. See \cite{brenti-open-problems} for a further summary of work on the Combinatorial Invariance Conjecture.

In recent work \cite{Blundell, davies}, Blundell, Buesing, Davies, Veli\v{c}kovi\'{c}, and Williamson used artificial intelligence techniques to conjecture a recurrence for the $P_{u,v}(q)$ for $S_n$ based on combinatorial structures called \emph{hypercube decompositions} of $[u,v]$. The proof of their recurrence would resolve Conjecture~\ref{conj:cic} for $S_n$. The authors developed a geometric approach towards their conjecture, but were unable to complete it. 

The \emph{$\tR$-polynomials} $\tR_{u,v}(q)$ for $u \leq v$ are an alternative family of polynomials which determine and are determined by the family of $P_{u,v}(q)$ (see Section~\ref{sec:prelim-kl}). In particular, the Combinatorial Invariance Conjecture may equivalently be stated as follows.

\begin{conjecture}[Equivalent formulation of Conj.~\ref{conj:cic}]
\label{conj:R-cic}
The $\tR$-polynomials are combinatorial. In other words, for $u,v \in W$ and $u',v' \in W'$, if $[u,v] \cong [u',v']$, then
\[\tR_{u,v}(q)=\tR_{u',v'}(q).\]
\end{conjecture}

In this article\footnote{An extended abstract describing part of this work appears \cite{fpsac-abstract} in the proceedings of FPSAC 2023.} we introduce the modified notion of a \emph{strong hypercube decomposition} $I$ of an interval $[u,v] \subset S_n$ (see Definition~\ref{def:strong-hcd}) and use this to formulate a conjectural combinatorial recurrence instead for the $\tR$-polynomials\footnote{In recent independent work \cite{Gurevich-Wang} Gurevich and Wang have interpreted the recurrence of Blundell et al. in terms of a new family of polynomials closely related to the $R$-polynomials. The forthcoming work \cite{Brenti-Marietti} of Brenti and Marietti also interprets hypercube decompositions in terms of $R$-polynomials in a different way.}. Our recurrence has the advantage of simplicity and amenability to combinatorial techniques, compared to the one in \cite{Blundell}. We define a polynomial $\tH_{u,v,I}(q)$ (see Equation~\ref{eq:H-definition}), which may be computed solely from the pair $(I,[u,v])$ up to poset isomorphism and from inductive knowledge of $\tR_{u,x}(q)$ for smaller intervals, and conjecture:

\begin{conjecture}
\label{conj:inequality}
For any strong hypercube decomposition $I \subset [u,v]$ we have
\begin{equation}
\label{eq:H-R-inequality}
\tH_{u,v,I} \geq \tR_{u,v},
\end{equation}
where we compare polynomials coefficientwise.
\end{conjecture}

\begin{remark}
Conjecture~\ref{conj:inequality} has been verified for all intervals $[u,v] \subset S_6$ using SageMath. The smallest instance of strict inequality in (\ref{eq:H-R-inequality}) occurs for $u=132546, v=651234$ and $I=[u,z]$, where $z=612345$.
\end{remark}

There always exists a \emph{standard} hypercube decomposition $I_{\mathrm{st}} \subset [u,v]$ (see \cite[Thm.~3.7]{Blundell}) and we show in \Cref{thm:existsequality} that equality holds in (\ref{eq:H-R-inequality}) for $I_{\mathrm{st}}$. Thus Conjecture~\ref{conj:inequality} would imply that
\[
\tR_{u,v} = \min_{I} \tH_{u,v,I},
\]
where $I$ ranges over all strong hypercube decompositions. Conjecture~\ref{conj:inequality} implies Conjectures~\ref{conj:cic} and \ref{conj:R-cic} since it gives a method for computing $\tR_{u,v}(q)$ which depends only on the isomorphism type of $[u,v]$.

We prove Conjecture~\ref{conj:inequality}, and thus Conjecture~\ref{conj:R-cic}, for a class of intervals we call \emph{elementary}.

\begin{definition}
\label{def:simple-elementary}
We say the Bruhat interval $[u,v] \subset S_n$ is \emph{simple} if the roots $e_i - e_j$ corresponding to the covers $u \lessdot (ij)u \leq v$ are linearly independent. We say $[u,v] \subset S_n$ is \emph{elementary} if $[u,v] \cong [u',v']$ for some simple interval $[u',v'] \subset S_{n'}$. 
\end{definition}

\begin{example} \text{}
\begin{enumerate}
    \item Any lower interval $[e,v]$ is simple, since the simple roots $e_i-e_{i+1}$, for $i=1,\ldots,n-1$ corresponding to possible upper covers of $e$ in $[e,v]$ are linearly independent.
    \item The interval $[u,v]=[1324, 4231]$ is isomorphic to a 4-hypercube. It is not simple, but is isomorphic to the simple interval $[12345678,21436587] \subset S_8$; thus $[u,v]$ is elementary. 
\end{enumerate}
\end{example}

\begin{theorem}
\label{thm:main-theorem}
Suppose that $[u,v] \subset S_n$ is elementary and $[u,v] \cong [u',v'] \subset S_m$. Then \[\tR_{u,v}(q)=\tR_{u',v'}(q). \]
\end{theorem}

Our proof of Theorem~\ref{thm:main-theorem} is the first application of hypercube decompositions to prove a new instance of combinatorial invariance. Previous results have used \emph{special matchings} \cite{Brenti-lower-intervals, Lower-intervals-general-type} or classifications of small intervals \cite{affine-A2, Incitti}. This new method is promising in light of Conjecture~\ref{conj:inequality} and the fact that all intervals $[u,v] \subset S_n$ have hypercube decompositions.  This means that arbitrary intervals could be within reach of our method.

Our result is new even in the case $[u,v] \cong [e,v']$ is isomorphic to a lower interval, as the special matching recurrence for $\tR$-polynomials used in \cite{Brenti-lower-intervals, Lower-intervals-general-type} is not known to hold in that case; du Cloux \cite[\S1.3]{DuCloux} also specifically cites this case as beyond his methods. The class of elementary intervals is much richer than this, however. Example~\ref{ex:no-special-matching} gives a simple interval which has no special matchings at all, and thus is not isomorphic to a lower interval, and seems beyond the reach of previous techniques.

\begin{example}
\label{ex:no-special-matching}
If $u=21354$ and $v=52341$, then $[u,v]$ is simple and has no special matchings.
\end{example}

In Section~\ref{sec:prelim} we recall background on Bruhat order and Kazhdan--Lusztig polynomials. In Section~\ref{sec:hypercubes} we introduce strong hypercube decompositions and the $\tH$-polynomials and study their relationship to the $\tR$-polynomials. In Section~\ref{sec:dc-ideals} we show that hypercube decompositions in simple intervals take a particularly nice form; when combined with facts about $I_{\mathrm{st}}$ from Section~\ref{sec:standard}, this implies Theorem~\ref{thm:main-theorem}.

\section{Preliminaries}
\label{sec:prelim}
The background material in this section can be found in Bj\"{o}rner--Brenti \cite{Bjorner-Brenti}.

\subsection{Bruhat order}
We view the symmetric group $S_n$ as a Coxeter group with simple reflections $S=\{s_i \mid i=1,\ldots,n-1\}$ where $s_i=(i \: i{+}1)$. We write $T$ for the set of \emph{reflections} $(ij)$, the $S_n$-conjugates of the simple reflections. The \emph{root} corresponding to a reflection $t=(ij)$ is $\alpha_t=e_i-e_j,$ where $i<j$ and $e_i$ denotes the $i$-th standard basis vector in $\mathbb{R}^n$. A \emph{reflection subgroup} $W' \subset S_n$ is a subgroup generated by a set of reflections. Any reflection subgroup of $S_n$ is a product of subgroups $S_K$, the permutations of $K \subset \{1,\ldots,n\}$.

For $u \in S_n$, the \emph{length} $\ell=\ell(u)$ is the smallest number of terms in an expression $u=s_{i_1}\cdots s_{i_{\ell}}$; such an expression of minimal length is a \emph{reduced expression}. We write $\ell(u,v)$ for $\ell(v)-\ell(u)$.

\begin{definition} 
The \emph{Bruhat graph} $\Gamma$ is the edge-labelled directed graph on vertex set $S_n$ with edges $u \xrightarrow{t} tu$ whenever $t \in T$ and $\ell(u)<\ell(tu)$. The \emph{Bruhat order} $\leq$ is the partial order on $S_n$ with $u \leq v$ whenever there is a directed path from $u$ to $v$ in $\Gamma$. We let $\Gamma(u,v)$ denote the induced subgraph of the Bruhat graph on the elements of the Bruhat interval $[u,v]$. 
\end{definition}

\begin{remark}
\label{rem:bg-combinatorial-labels-not}
By a result of Dyer \cite{Dyer-bruhat-graph}, the underlying unlabelled directed graph of $\Gamma(u,v)$ is combinatorial. The labelled version, however, is not.
\end{remark}

\subsection{Kazhdan--Lusztig polynomials}
\label{sec:prelim-kl}

The \emph{descents} $\Des(u)$ of $u$ are those $s \in S$ such that $u>us$. 

\begin{definition}[Kazhdan--Lusztig \cite{Kazhdan-Lusztig-1}]
\label{def:kl-polys}
Define polynomials $R_{u,v}(q) \in \mathbb{Z}[q]$ by:
\[R_{u,v}(q) = \begin{cases} 0, &\text{ if $u \not \leq v$.} \\ 1, &\text{ if $u=v$.} \\ R_{us,vs}(q), &\text{ if $s \in \Des(u) \cap \Des(v)$.} \\ qR_{us,vs}(q) + (q-1)R_{u,vs}(q), &\text{ if $s \in \Des(v)\setminus \Des(u)$.}\end{cases}\]
Then there is a unique family of polynomials $P_{u,v}(q) \in \mathbb{Z}[q],$ the \emph{Kazhdan--Lusztig polynomials} satisfying $P_{u,v}(q)=0$ if $u \not \leq v$, $P_{v,v}(q)=1$, and such that if $u<v$ then $P_{u,v}$ has degree at most $\frac{1}{2}(\ell(u,v)-1)$ and 
\begin{equation}
\label{eq:P-definition}
q^{\ell(u,v)}P_{u,v}(q^{-1})=\sum_{a \in [u,v]} R_{u,a}(q)P_{a,v}(q).
\end{equation}
\end{definition}

\subsection{Reflection orders and $\tR$-polynomials}

\begin{proposition}
\label{prop:R-tilde-exist}
Let $u,v \in S_n$, then there exists a unique polynomial $\tR_{u,v}(q) \in \mathbb{N}[q]$ such that
\[R_{u,v}(q)=q^{\frac{\ell(u,v)}{2}}\tR_{u,v}(q^{\frac{1}{2}}-q^{-\frac{1}{2}}).\]
\end{proposition}

Given a Bruhat interval as an abstract poset and the $\tR$-polynomials for all subintervals, Proposition~\ref{prop:R-tilde-exist} can be used to recover the $R$-polynomials for all subintervals, and (\ref{eq:P-definition}) can be used to recursively compute the $P$-polynomials. Thus a proof that the $\tR$-polynomials are combinatorial would resolve Conjecture~\ref{conj:cic}. \Cref{thm:tilde-R-counts-increasing-paths} gives a formula for the $\tR$-polynomials which is not combinatorial but which will be our main tool for working with them.

\begin{definition}
    A \emph{reflection order} is a total ordering $\prec$ on the reflections $T$ such that for all $1\leq a<b<c \leq n$ we have either
    \[ (a\:b) \prec (a\:c) \prec (b\:c) \qquad \text{or} \qquad (b\:c) \prec (a\:c) \prec (a\:b). \]
    A directed path $x_1\xrightarrow{t_1} x_2\xrightarrow{t_2} \cdots \xrightarrow{t_{k-1}} x_k$ in the Bruhat graph $\Gamma$ is said to be \emph{increasing} with respect to a reflection order $\prec$ if 
    \[ t_1 \prec t_2 \prec \cdots \prec t_{k-1}.\] 
    Denote the set of increasing paths in $\Gamma$ from $u$ to $v$ by $\Gamma_\prec(u,v)$.
\end{definition}

\begin{theorem}[Dyer \cite{Dyer1987}]
\label{thm:tilde-R-counts-increasing-paths}
    Let $\prec$ be a reflection order on $T$. Then
    \[ \tR_{u,v}(q) = \sum_{\gamma \in \Gamma_{\prec}(u,v)} q^{\ell(\gamma)}, \]
    where $\ell(\gamma)$ is the number of edges in the path $\gamma$.
\end{theorem}

\section{Strong hypercube decompositions}
\label{sec:hypercubes}
In this section we introduce a refinement of the notion of a hypercube decomposition; see \Cref{def:bbdvw-hcd} for the original definition from \cite{Blundell}. Of particular importance will be various operations on ``diamonds,'' some of which were discussed in \cite{Patimo-q-coefficient}.

\begin{definition}
Let $\cH_n$ be the \emph{hypercube graph} on $n$ elements. This is the directed graph which is the Hasse diagram of the Boolean lattice with $n$ atoms. An $n$-\emph{hypercube} is a subgraph of $\Gamma$ isomorphic to $\cH_n$.  A 2-hypercube is called a \emph{diamond}. This is a subgraph of $\Gamma$ having (distinct) vertices $\{x_1,x_2,x_3,x_4\}$ with directed edges $x_1 \to x_2\to x_4$ and $x_1\to x_3 \to x_4$. 

Given a sequence of two edges $x_1\to x_2\to x_4$ in $\Gamma$, their \emph{diamond flip} is the unique\footnotemark pair of edges $x_1\to x_3\to x_4$ such that $\{x_1,x_2,x_3,x_4\}$ is a diamond as above.

Given sets $X \subset Y \subset S_n$, we say $X$ is \emph{diamond-closed} in $Y$ if whenever $X$ contains three vertices of a diamond contained in $Y$, it also contains the fourth. 
\footnotetext{The two edges generate a dihedral reflection subgroup of $S_n$, which must be $S_2 \times S_2$ or $S_3$, so it suffices to check existence and uniqueness for these. Diamond flipping can be defined for any Coxeter group, but outside of simply-laced type the uniqueness fails and the definition is more subtle.}
\end{definition}

\begin{lemma}\label{lem:diamondfliporder}
    Let $\prec$ be a reflection order. Given a diamond with the following edge labels, such that $t_1\prec t_1'$:
    \[
\begin{tikzcd}[arrows={start anchor=center, end anchor=center,-,rmarrow=>}]
	& \phantom{a} & \\
	\ar[ur,"t_2"]& &\ar[ul,"t_2'"'] \\
	& \ar[ul,"t_1"]\ar[ur,"t_1'"'] &
\end{tikzcd},
    \]
    we have
    \[ t_1\prec t_2\succ t_2' \prec t_1'. \]
    In particular, the diamond flip of an increasing pair of edges is a decreasing pair of edges.
\end{lemma}
\begin{proof}
    By \cite[Lemma 3.1]{Dyer-bruhat-graph}, the reflection subgroup generated by $t_1,t_2,t_1',t_2'$ is dihedral. The restriction of a reflection order to a reflection subgroup is a reflection order. Hence it is enough to show the claim for diamonds in $S_2\times S_2$ and in $S_3$, which is a straightforward check.
\end{proof}

\begin{lemma}\label{thm:chaintohypercube}
    Any path $\gamma$ from $x$ to $y$ in $\Gamma$ is contained in at most one hypercube with bottom vertex $x$ and top vertex $y$.
\end{lemma}
\begin{proof}
    We will show that one can recover the hypercube $\cH$ containing $\gamma$ as the minimal set of edges containing $\gamma$ which is closed under taking all possible diamond flips. Indeed, note that any hypercube is closed under taking diamond flips. It remains to show that a maximal chain in $\cH_n$ generates $\cH_n$ under diamond flips. Viewing the vertices of $\cH_n$ as subsets of $Y=\{y_1,\ldots,y_n\}$, we may take without loss of generality 
    \[\gamma = (\varnothing\to \{y_1\}\to \{y_1,y_2\} \to \cdots \to Y).\]
    So we may think of maximal chains as total orders of $y_1,\ldots, y_n$. 
    Replacing the pair of edges $\{y_1,\ldots,y_{k-1}\} \to \{y_1,\ldots, y_k\} \to \{y_1,\ldots,y_{k+1}\}$ with their diamond flip corresponds to swapping $y_k$ with $y_{k+1}$ in the total order. Clearly we can realize every total order with this process, and hence we will find every maximal chain (and in particular, every edge).
\end{proof}

Consider a fixed Bruhat interval $[u,v]$. 

\begin{definition}
    Let $I\subset [u,v]$ be an order ideal and let $x\in I$. Define
    \[ \cY_x = \{ y\in [u,v]\setminus I \mid x\to y\} \]
    and 
    \[ \cA_x = \{ Y\subset \cY_x \mid Y\text{ is an antichain in Bruhat order} \}. \]
    The collection $\cA_x$ is naturally a poset under inclusion order. We also view it as a directed graph where $Y_1 \to Y_2$ if $Y_1\subset Y_2$ and $|Y_2\setminus Y_1|=1.$ Call the elements of $\cA_x$ \emph{antichains over $x$} (relative to $I$).

    A \emph{strong hypercube cluster} at $x$ (relative to $I$) consists of a function $\theta_x:\cA_x \to [u,v]$, called the \emph{hypercube map}, satisfying the following four properties: 
    \begin{enumerate}
        \item[(HC1)] $\theta_x(\varnothing) = x$.
        \item[(HC2)] If $y\in\cY_x$, then $\theta_x(\{y\}) = y$.
        \item[(HC3)] If $Y_1\to Y_2$, then $\theta_x(Y_1)\to \theta_x(Y_2)$.
        \item[(HC4)] If $|Y|=|Y'|=|Y\cap Y'|+1$ and we have a subgraph of $\Gamma(u,v)$ of the form
        
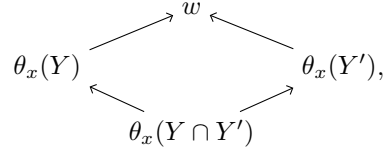
\[\begin{tikzcd}[cramped,sep=small] & w & \\ \theta_x(Y)\ar[ur]& &\theta_x(Y')\ar[ul] \\ &\theta_x(Y\cap Y')\ar[ul]\ar[ur]& \end{tikzcd},\]  then $Y\cup Y'$ is an antichain and $\theta_x(Y\cup Y')= w$.
    \end{enumerate}
\end{definition}

Hypercube clusters are so-named because the hypercube map sends intervals in $\cA_x$ to hypercubes in the Bruhat graph, so the existence of a strong hypercube cluster at $x$ means that there are lots of hypercubes in $\Gamma(u,v)$ which have bottom vertex $x$.  Note that for a fixed order ideal $I$, there is at most one strong hypercube cluster at $x$ relative to $I$. In fact, if there is a strong hypercube cluster then there is at most one hypercube at $x$ with a given set of bottom edges. This follows from the following lemma.

\begin{lemma}\label{thm:bottomedges}
    Let $I$ be an order ideal in $[u,v]$ and let $x\in I$ have a strong hypercube cluster relative to $I$, with hypercube map $\theta_x:\cA_x\to [u,v]$. Let $\cI\subset \cA_x$ be any order ideal. Then for any function $\phi:\cI\rightarrow \Gamma(u,v)$ satisfying (HC1-3), we have that $\phi$ is the restriction of $\theta_x$ to $\cI$.
\end{lemma}
\begin{proof}
    It is enough to show the claim for $\cI = [\varnothing,Y]$. The proof is by induction on $|Y|$. Indeed, the case $|Y|=1$ follows since $\phi$ is determined on singletons. And if $|Y|>1$ then for any two distinct elements $y_1,y_2\in Y$ we have a subgraph of $\Gamma(u,v)$ of the form 
    
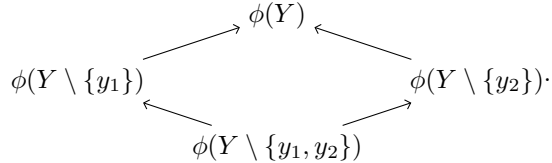
\[\begin{tikzcd}[cramped,sep=small] & \phi(Y) & \\ \phi(Y\setminus\{y_1\})\ar[ur]& &\phi(Y\setminus\{y_2\})\ar[ul] \\ &\phi(Y\setminus\{y_1,y_2\})\ar[ul]\ar[ur]& \end{tikzcd}.\]
    Hence by induction and (HC4), we can conclude $\phi(Y)=\theta_x(Y)$.
\end{proof}

\begin{definition}
\label{def:strong-hcd}
A \emph{strong hypercube decomposition} of $[u,v]$ consists of an order ideal $I$ satisfying the following properties:
\begin{enumerate}
    \item[(HD1)] $I = [u,z]$ for some $z\in [u,v]$.
    \item[(HD2)] $I$ is diamond-closed in $[u,v]$.
    \item[(HD3)] For each $x$ in $I$, there is a strong hypercube cluster $\theta_x : \cA_x \to [u,v]$, relative to $I$.
\end{enumerate}
\end{definition}

\begin{remark}
Hypercube decompositions (see \Cref{def:bbdvw-hcd}) were introduced in \cite{Blundell, davies}. \Cref{thm:bottomedges} shows that our strong hypercube decompositions are always hypercube decompositions in their sense, and we do not know of any hypercube decompositions which are not strong. In particular, we will show in \Cref{sec:standard} that the standard hypercube decomposition is strong.
\end{remark}

Using the data associated to a hypercube decomposition, we now define a polynomial which is conjecturally related to the $\tR$-polynomial by Conjecture~\ref{conj:inequality}:
\begin{equation}
\label{eq:H-definition}
\tH_{u,v,I}(q) \coloneqq \sum_{x\in I}\sum_{\substack{Y\in \cA_x \\ \theta_x(Y)=v}} q^{|Y|}\tR_{u,x}(q).
\end{equation}
Here the inner sum is over antichains $Y$ over $x$ such that $\theta_x(Y)=v$. Such antichains are necessarily maximal elements in $\cA_x$. 

\begin{remark}\label{rem:isomH}
    The data of a hypercube decomposition is purely combinatorial. More precisely, if we have a poset isomorphism $\phi:[u,v]\to [u',v']$ and $I\subset [u,v]$ is a strong hypercube decomposition, then $\phi(I)$ is a strong hypercube decomposition of $[u',v']$. Furthermore, $\phi(\theta_x(Y)) = \theta_{\phi(x)}(\phi(Y))$ for all $Y\in \cA_x$, by \Cref{thm:bottomedges}. Write $I'=\phi(I)$. It follows that
    \[ \tH_{u',v',I'}(q) = \sum_{x\in I}\sum_{\substack{Y\in \cA_{x} \\ \theta_{x}(Y)=v}} q^{|Y|}\tR_{u',\phi(x)}(q).  \]
    Hence if we know $\tR_{u,x}(q) = \tR_{u',\phi(x)}(q)$ for all $x\in I$ (which is predicted to be true by the combinatorial invariance conjecture) then it follows that $\tH_{u,v,I}(q)=\tH_{u',v',I'}(q)$. 
\end{remark}

To prove a connection between the $\tR$- and $\tH$-polynomials, we will need to consider certain specially constructed reflection orders. Fix an order ideal $I$ in $[u,v]$ and consider the following properties of a reflection order $\prec$ on $T$.
\begin{itemize}
    \item[(E1)] For each $x\in I$, if $x\xrightarrow{t_1} y_1$ and $x\xrightarrow{t_2} y_2$, and $y_1 \in I$ and $y_2 \in [u,v]\setminus I$, then $t_1\prec t_2$.
    \item[(E2)] For each $x\in I$, if $y_1\xrightarrow{t_1} x$ and $x\xrightarrow{t_2} y_2$, and $y_2\in [u,v]\setminus I$, then $t_1\prec t_2$.
    \item[(E)] If $x_1\xrightarrow{t}x_2$ with $x_1,x_2\in I$, and $x\xrightarrow{t'}y$ with $x\in I$ and $y\in [u,v]\setminus I$, then $t\prec t'$.
\end{itemize}

\begin{remark}
    Property (E) implies properties (E1) and (E2). Furthermore, if $I$ satisfies (HD2), then, by \Cref{lem:diamondfliporder}, (E1) implies (E2). There are pairs $(I,[u,v])$ and $(I',[u',v'])$ which are hypercube decompositions and are isomorphic as intervals equipped with a subset, but for which $(I,[u,v])$ satisfies (E1) (respectively (E2), (E)) and $(I',[u',v'])$ does not.
\end{remark}
 Our main result on $\tH$-polynomials is the following. 
\begin{theorem}\label{thm:RHinequality}
    Let $I$ be a strong hypercube decomposition of $[u,v]$.
    \begin{enumerate}[(i)]
        \item If there exists a reflection order satisfying (E1), then $\tR_{u,v}\leq \tH_{u,v,I}$.
        \item If there exists a reflection order satisfying (E2), then $\tR_{u,v}\geq \tH_{u,v,I}$.
        \item If there exists a reflection order satisfying (E), then $\tR_{u,v}= \tH_{u,v,I}$.
    \end{enumerate}
\end{theorem}
\begin{remark}
    In fact, we prove \Cref{thm:RHinequality} using only that $I$ is an order ideal of $[u,v]$ satisfying (HD3). However, properties (HD1) and (HD2) will be critical in \Cref{sec:diamond-closed} and thereby in proving combinatorial invariance.
\end{remark}

    We break the proof of \Cref{thm:RHinequality} into two statements relating increasing paths to hypercube clusters, one for each inequality direction. Using the description of $\tR$ as a path-counting function (\Cref{thm:tilde-R-counts-increasing-paths}), these lemmas immediately imply the theorem. We start with the easier direction.

    \begin{lemma}\label{thm:pathsinhypercube}
        Fix a strong hypercube decomposition $I$ of $[u,v]$ and let $\prec$ be a reflection order. Then for each $x\in I$ and antichain $Y\in \cA_x$, there exists a unique ordering
        \[ Y = \{y_1,y_2,\ldots, y_k\}  \]
        such that 
        \[ x \to \theta_x(y_1) \to \theta_x(y_1,y_2) \to \cdots \to \theta_x(y_1,\ldots,y_k) \]
        is an increasing path in $\Gamma_\prec(x,\theta_x(Y))$.
    \end{lemma}
    \begin{proof}
        Let $\cH$ be the hypercube which is the image of $[\varnothing,Y]$ under $\theta_x$. 
        If $w$ is a vertex of $\cH$, define
        \[\cH_w \coloneqq \{t\mid w\xrightarrow{t} w' \text{ is an edge of }\cH\}.\]
        We claim that if $w\xrightarrow{t}w'$, and $t$ is the $n$-th smallest element ($n=1,2,\ldots$) of $\cH_w$, then $t$ is the $n$-th smallest element of $\{t\}\cup \cH_{w'}$. This implies that if
        \[ w_0=x \xrightarrow{t_1} w_1 \xrightarrow{t_2} \cdots \xrightarrow{t_r} w_r \]
        is an increasing path in $\cH$, such that $t_i$ is the $n_i$-th smallest element of $\cH_{w_{i-1}}$ for all $i$, then $n_1\leq n_2\leq \cdots \leq n_{r}$. Notice that the lemma follows from this fact, since when $r=k$ we have that $|\cH_{w_{r-1}}| = 1$ which is only possible if $n_{r}=1$. Hence an increasing path can reach the top of $\cH$ if and only if it is the ``greedy'' path built from bottom to top by choosing the smallest possible edge label at each branch.

        We now prove the claim. For any $t'\in \cH_{w'}$, consider the diamond below.
        
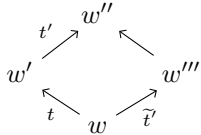
\[\begin{tikzcd}[cramped,sep=small] & w'' & \\ w'\ar[ur,"t'"]& &w'''\ar[ul] \\ &w\ar[ul,"t"]\ar[ur,"\widetilde{t}'"']& \end{tikzcd}\]
        The map $\cH_{w'}\to \cH_{w}$ sending $t'\mapsto \widetilde{t}'$ is injective, with image $\cH_{w}\setminus\{t\}$. This is due to a special property of hypercubes, which is that any pair of outgoing edges from a given vertex is the bottom half of a unique diamond. Furthermore, \Cref{lem:diamondfliporder} shows that $t\prec \widetilde t'$ if and only if $t\prec t'$. In particular, if $t$ is the $n$-th smallest element of $\{t\}\cup \cH_{w'}$ then $t$ is also the $n$-th smallest element of $\cH_{w}$.
    \end{proof}
    \begin{remark}
        This is essentially a proof that reflection orders give EL-labelings of hypercubes, even for hypercubes which are not Bruhat intervals.
    \end{remark}

    \begin{corollary}
        In the situation of \Cref{thm:pathsinhypercube}, if $\prec$ satisfies (E2), then 
        \[\tR_{u,v}\geq \tH_{u,v,I}.\]
    \end{corollary}
    \begin{proof}
        The right hand side counts increasing paths $\gamma$ from $u$ which end at an element $x\in I$, equipped with an antichain $Y$ over $x$ for which $\theta_x(Y)=v$. These are weighted by $q^{\ell(\gamma)+|Y|}$. We need to construct a weight-preserving injection from these objects into $\Gamma_\prec(u,v)$. To do so, simply take the pair $(\gamma,Y)$ to the path $\widetilde\gamma$ from $u$ to $v$ which begins with $\gamma$ and ends with the path from \Cref{thm:pathsinhypercube}. The length of $\widetilde\gamma$ is $\ell(\gamma)+|Y|$, so the construction is weight-preserving. We need to show that the construction gives an increasing path and that it is an injection.

        Let $x'\xrightarrow{t_1} x$ be the last edge of $\gamma$ (if this does not exist then we are done). And let $x\xrightarrow{t_2} \theta_x(y_1)$ be the first edge of the path from the lemma. To show that $\widetilde\gamma$ is an increasing path, it suffices to show that $t_1\prec t_2$. But this is guaranteed by property (E2).

        Now we need to show that the assignment $(\gamma,Y)\mapsto \widetilde\gamma$ is an injection. First note that we can recover $x$ from $\widetilde\gamma$ as the last vertex in $I$ on the path, and $\gamma$ as the initial section of $\widetilde\gamma$ ending at $x$. 
        Then \Cref{thm:chaintohypercube} lets us recover $Y$ as the set of atoms of the hypercube generated by the remainder of the path under diamond flips.
    \end{proof}

    \begin{lemma}\label{thm:E1}
        If $I$ is a strong hypercube decomposition of $[u,v]$ and $\prec$ is a reflection order satisfying (E1), then any path in $\Gamma_\prec(u,v)$ ends with a sequence of the form
        \[ \cdots \to x \to \theta_x(y_1) \to \theta_x(y_1,y_2) \to \cdots \to \theta_x(y_1,y_2,\ldots,y_k) \]
        for some $x\in I$ and $\{y_1,\ldots,y_k\}\in \cA_x$.
    \end{lemma}
    \begin{proof}
        Let $\gamma\in \Gamma_\prec(u,v)$ be an increasing path and let $x$ be the last vertex of $\gamma$ which is in $I$. Let
        \[ w_0=x \xrightarrow{e_1} w_1 \xrightarrow{e_2} \cdots \xrightarrow{e_n} w_n \]
        be the vertices and edge labels on the tail of $\gamma$ starting at $x$. We will induct on $n$ to show that $w_{n-1}\to w_n$ is in the image of $\theta_x$. We may assume by induction that there is some $Y=\{y_1,\ldots, y_{n-1}\}$ in $\cA_x$ such that $\theta_x(y_1,\ldots, y_i) = w_i$ for all $1\leq i < n$. Our strategy will be to determine the element $y_n\in \cY_x$ using a similar idea as in the proof of \Cref{thm:chaintohypercube}: by systematically diamond flipping $\gamma$ from top to bottom, the new outgoing edge from $x$ that we create should be $x\to y_n$. 

        Let us introduce some notation to keep track of the process. The fully diamond flipped path will be labeled
        \[ w_0'=x \xrightarrow{e_{1}'} w_1' \xrightarrow{e_2'} \cdots \xrightarrow{e_n'}w_n'=w_n. \]
        Set $f_n=e_n$. We (reverse) inductively construct the flipped path by defining the edge pair $w_{i-1}\xrightarrow{f_i} w_i'\xrightarrow{e_{i+1}'}w_{i+1}'$ to be the diamond flip of $w_{i-1}\xrightarrow{e_i}w_i\xrightarrow{f_{i+1}} w_{i+1}'$. The following illustrates one of the diamonds:
        
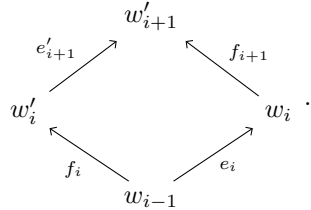
\[
	\begin{tikzcd}
		& w'_{i+1}& \\
		w'_{i}\ar[ur,"e_{i+1}'"]& & w_{i}\ar[ul,"f_{i+1}"'] \\
		& w_{i-1}\ar[ur,"e_i"']\ar[ul,"f_i"] & 
	\end{tikzcd}.
	\]

        Because by assumption $e_1\prec e_2 \prec \cdots \prec e_n$, we can show inductively that $e_i \prec f_i$ using \Cref{lem:diamondfliporder}. We find in particular that $e_1\prec f_2$; then applying \Cref{lem:diamondfliporder} once more gives $f_1\succ e_1$. Now if $w_1'$ were in $I$, then (E1) plus this last inequality would imply that $w_1$ is in $I$. But this is a contradiction, since $w_1=y_1\in \cY_x$. Hence $w_1'$ is not in $I$. But that means $w_1'\in \cY_x$. It remains to see that $Y\cup\{w_1'\} \in \cA_x$ and that $\theta_x(Y\cup\{w_1'\})=w_n$.

        We will show by induction on $i$ that $y_1,\ldots,y_i,w_1'$ are pairwise incomparable and that $\theta_x(y_1,\ldots,y_i,w_1') = w_{i+1}'$. The base case is clear. If $i\geq 1$, then the inductive hypothesis lets us write our diamond above as
        \[
	\begin{tikzcd}
		& w'_{i+1}& \\
		\theta_x(y_1,\ldots,y_{i-1},w_1')\ar[ur,"e_{i+1}'"]& & \theta_x(y_1,\ldots,y_i)\ar[ul,"f_{i+1}"'] \\
		& \theta_x(y_1,\ldots,y_{i-1})\ar[ur,"e_i"']\ar[ul,"f_i"] & 
	\end{tikzcd}.
	\]
        Now apply (HC4) and we are done.
    \end{proof}

    \begin{corollary}
        In the situation of \Cref{thm:E1}, we have 
        \[\tR_{u,v}\leq \tH_{u,v,I}.\]
    \end{corollary}
    \begin{proof}
        This is the converse of the setup in the previous corollary. We would like to set up an injection from $\widetilde\gamma\in \Gamma_\prec(u,v)$ into the set of triples $(\gamma,x,Y)$ of increasing paths $\gamma$ ending at a vertex $x\in I$ and $Y\in \cA_x$ such that $\theta_x(Y)=v$. Furthermore, this injection should take the length $\ell(\widetilde\gamma)$ to the value $\ell(\gamma)+|Y|$. As one might imagine, we take $x$ to be the last vertex of $I$ on $\widetilde\gamma$, take $\gamma$ to be the truncation of $\widetilde\gamma$ up to $x$, and use \Cref{thm:E1} to find $Y \in \cA_x$ such that the remainder of $\widetilde\gamma$ stays in the image of $\theta_x$. By \Cref{thm:chaintohypercube}, we can recover the set $Y$ constructed in the \Cref{thm:E1} from the path. 
        
        This gives us our triple $(\gamma,x,Y)$. That the assignment $\widetilde\gamma\mapsto (\gamma,Y)$ is an injection follows from the uniqueness part of \Cref{thm:pathsinhypercube}. And we have $\ell(\widetilde\gamma)-\ell(\gamma) = |Y|$ by the description in \Cref{thm:E1}, so we are done.
    \end{proof}

\section{Diamond-closed order ideals}
\label{sec:diamond-closed}
In this section we show that diamond-closed order ideals in simple intervals are given by intersections with cosets of reflection subgroups, so that $\tR=\tH$ for these intervals.

\label{sec:dc-ideals}
Write $A_{uv}$ for the set of \emph{atoms} (elements covering $u$) in $[u,v]$, $T_{uv}$ for the corresponding set $\{au^{-1} \mid a \in A_{uv}\}$ of reflections, and $\Phi_{uv}$ for the corresponding set of roots. For $I \subset [u,v]$, define $A_I\coloneqq A_{uv} \cap I, T_I\coloneqq \{au^{-1} \mid a \in A_I\},$ and $\Phi_I$ to be the set of roots corresponding to $T_I$.

\begin{definition}
    For any $X \subset Y \subset S_n$, there is a unique smallest set $\dc_Y(X) \subset Y$ containing $X$ and diamond-closed in $Y$:
\[\dc_Y(X) \coloneqq \bigcap_{\substack{D: X \subset D \subset Y \\ D \text{ diamond-closed}}} D,\]
the \emph{diamond closure} of $X$ in $Y$.
\end{definition}

\begin{lemma}
\label{lem:dc-ideal-is-dc-of-atoms}
Let $I \subset [u,v]$ be a nonempty diamond-closed order ideal. Then \[
I=\dc_{[u,v]}(\{u\} \cup A_I).
\]
\end{lemma}
\begin{proof}
Let $D\coloneqq \dc_{[u,v]}(\{u\} \cup A_I)$; since $I$ is diamond-closed and contains $\{u\} \cup A_I$, we have $D \subset I$. Suppose they are unequal, and let $y$ be a minimal element of $I \setminus D$. We have $\ell(u,y) \geq 2$, since $I$ and $D$ both have atoms $A_I$. Thus we may choose $w \in [u,v]$ such that $w \leq y$ and $\ell(w,y)=2$. Any height-two interval in Bruhat order is a diamond, so $[w,y]=\{w,x,x',y\}$ with $w \lessdot x,x' \lessdot y$. Since $I$ is an order ideal, and by the minimality of $y$, we know $w,x,x' \in I \cap D$. But this contradicts the diamond-closedness of $D$, thus $I=D$.
\end{proof}

\begin{lemma}
\label{lem:reflection-subgroup-is-dc}
Let $W' \subset S_n$ be a reflection subgroup and $u \in S_n$, then $W'u$ is diamond-closed in $S_n$. 
\end{lemma}
\begin{proof}
Let $xu,t_1xu,t_2xu \in W'u$ be three elements of a diamond with fourth element $y=t_3t_1xu=t_4t_2xu$, where $t_1,t_2,t_3,t_4 \in T$. Then, by \cite[Lemma 3.1]{Dyer-bruhat-graph}, $W''=\langle t_1,t_2,t_3,t_4 \rangle$ is a dihedral group. Any dihedral reflection subgroup of $S_n$ is isomorphic to $S_2 \times S_2$ or to $S_3$, in either case, the distinct reflections $t_1,t_2$ generate $W''$. Since we have $t_1,t_2 \in W'$, this implies $t_3,t_4 \in W'$, so $y \in W'u$ and $W'u$ is diamond-closed.
\end{proof}

\begin{proposition}\label{thm:dc_simple_is_parabolic}
Let $I$ be a diamond-closed order ideal in the simple interval $[u,v]$, then $I=[u,v] \cap W'u$ for the reflection subgroup $W' \subset S_n$ generated by $T_I$.
\end{proposition}
\begin{proof}
Let $I'=[u,v] \cap W'u$, where $W'=\langle T_I \rangle$ is the reflection subgroup generated by $T_I$; we will show that $I=I'$. 

First, note that reflections $t_1,\ldots,t_k \in S_n$ are algebraically independent\footnote{meaning that $t_i$ is not in the subgroup generated by the other reflections, for all $i$} if the corresponding roots $\alpha_{t_1},\ldots,\alpha_{t_k}$ are linearly independent. Since $[u,v]$ is simple, we conclude that $A_{I'}=A_I$. By Lemma~\ref{lem:reflection-subgroup-is-dc} $W'u$ is diamond-closed, so $I'$ is diamond-closed in $[u,v]$. We now want to show that $I'$ is an order ideal, which will imply $I=I'$ by Lemma~\ref{lem:dc-ideal-is-dc-of-atoms}.

Let $Q$ be the convex polytope in $\mathbb{R}^n$ with vertices $\bs{x}=(x^{-1}(1),\ldots,x^{-1}(n))$ for $x \in [u,v]$; this is a \emph{Bruhat interval polytope} in the sense of \cite{Kodama-Williams, Tsukerman-Williams}. For $a=tu \in A_{uv}$, there is an edge of $Q$ between $\bs{u}$ and $\bs{a}$ parallel to $\alpha_t$. Consider the cone $C$ with base point $\bs{u}$ and rays $\Phi_{uv}$, clearly $Q \subset C$. Since $\Phi_{uv}$ is linearly independent, every subset of the rays spans a face, thus $F=C \cap (\bs{u}+V_I)$ is a face of $C$, where $V_I=\spn_{\mathbb{R}}(\Phi_I)$, and $F \cap Q$ is a face of $Q$.

Now, if $\bs{x}$ is a vertex of $F \cap Q$, then $x \in I'$, since there is a walk from $\bs{u}$ to $\bs{x}$ along edges of $F \cap Q$, each of which is parallel to a root lying in $V_I$. This gives rise to an expression $x=t_k \cdots t_1 u$ with each $t_i \in W'$. Conversely, if $x \in I'$, then $\bs{x} \in \bs{u}+V_I$, so $\bs{x}$ is a vertex of $F \cap Q$. Thus $\{\bs{x} \mid x \in I'\}$ is the vertex set of the face $F \cap Q$. By (the proof of) \cite[Theorem 4.1]{Tsukerman-Williams}, any face of a Bruhat interval polytope is the Bruhat interval polytope of a subinterval, so we must have $I'=[u,z]$ for some $z \in [u,v]$. In particular, $I'$ is an order ideal in $[u,v]$.
\end{proof}

The following is a useful lemma for constructing reflection orders.

\begin{lemma}\label{lem:constructreflorder}
    Fix $1\leq i \leq k$. If $t_1,t_2,\ldots,t_k$ are reflections such that the roots $\alpha_{t_1}, \alpha_{t_2},\ldots, \alpha_{t_k}$ are linearly independent, then there is a reflection order on $T$ satisfying the following properties:
    \begin{enumerate}[(i)]
        \item \(\displaystyle t_1 \prec t_2 \prec \cdots \prec t_k. \)
        \item The set of reflections $t$ such that $\alpha_t$ is in  $\spn_{\mathbb{R}}(\alpha_{t_1},\ldots, \alpha_{t_i})$ is an interval under $\prec$.
        \item If $t,t'$ are reflections such that $\alpha_t\in \spn_{\mathbb{R}}(\alpha_{t_1},\ldots, \alpha_{t_i})$ and \[\alpha_{t'}\in \spn_{\mathbb{R}}(\alpha_{t_1},\ldots, \alpha_{t_k})\setminus\spn_{\mathbb{R}}(\alpha_{t_1},\ldots, \alpha_{t_i}),\] then $t \prec t'$ if the coefficients of $\alpha_{t_{i+1}},\ldots,\alpha_{t_k}$ in $\alpha_{t'}$ are each non-negative.
    \end{enumerate} 
\end{lemma}
\begin{proof}
    \newcommand{\RR}{\mathbb{R}}
    Note that if $F_1,F_2:\RR^n\to \RR$ are any linear maps such that $F_2(\alpha_t)>0$ for all $t\in T$ and $t\mapsto \frac{F_1(\alpha_t)}{F_2(\alpha_t)}$ is injective, then ordering the reflections by
    \[ t_1\prec t_2 \text{ if and only if } \frac{F_1(\alpha_{t_1})}{F_2(\alpha_{t_1})} < \frac{F_1(\alpha_{t_2})}{F_2(\alpha_{t_2})} \]
    gives a reflection order.

    Choose positive roots $\alpha_{t_{k+1}},\ldots,\alpha_{t_{n-1}}$ so that $\{\alpha_{t_1}, \ldots, \alpha_{t_{n-1}}\}$ is a basis for the span of the roots of $S_n$. Now fix a choice of $F_2$ and positive values of $F_1(\alpha_{t_{j}})$ for $j>i$ so that
    \[ \frac{F_1(\alpha_{t_{i+1}})}{F_2(\alpha_{t_{i+1}})} < \ldots < \frac{F_1(\alpha_{t_k})}{F_2(\alpha_{t_k})}. \]
    By taking the values of $F_1(\alpha_{t_1}),\ldots,F_1(\alpha_{t_i})$ to be in a small enough neighborhood of zero, we can guarantee (ii); then (iii) is automatic. We can adjust the relative ordering of the values within the neighborhood of $0$ so that
    \[  \frac{F_1(\alpha_{t_{1}})}{F_2(\alpha_{t_{1}})} < \ldots < \frac{F_1(\alpha_{t_i})}{F_2(\alpha_{t_i})}. \]
    Hence we get (i).
\end{proof}

\begin{corollary}\label{thm:simpleRH}
    If $[u,v]$ is simple and $I$ is any strong hypercube decomposition, then $\tR_{u,v}(q) = \tH_{u,v,I}(q)$.
\end{corollary}
\begin{proof}
    Write $T_I = \{t_1,\ldots,t_i\}$ and $T_{uv} = \{t_1,\ldots,t_i,\ldots,t_k\}$. Let $\prec$ be the reflection order constructed by \Cref{lem:constructreflorder}. Using
    \Cref{thm:dc_simple_is_parabolic}, the edge labels of edges in $I$ are exactly the edge labels of $\{x\to y\mid x\in I,~y\in[u,v]\}$ which are in the span of $\Phi_I$. Furthermore, every edge in $\{x\to y\mid x\in I,~y\in[u,v]\}$ has a label $t$ such that $\alpha_t$ has non-negative coefficients on each root in $\Phi_{uv}\setminus \Phi_I$ (this is implied by the fact that $Q\subset C$, using the notation from the proof of the Proposition). Hence, the condition (E) is satisfied by $\prec$. So by \Cref{thm:RHinequality}, $\tR_{u,v}(q)=\tH_{u,v,I}(q)$.
\end{proof}

\section{Standard hypercube decompositions}\label{sec:standard}

The main result of this section is the following:
\begin{theorem}\label{thm:existsequality}
    In every Bruhat interval $[u,v]$ with $u<v$ there is an $I_{\mathrm{st}}\subsetneq [u,v]$ which gives a strong hypercube decomposition of $[u,v]$ and such that $\tR_{u,v}(q) = \tH_{u,v,I_{\mathrm{st}}}(q)$.
\end{theorem}

\begin{corollary}\label{thm:elementaryRH} 
    Let $[u,v]$ and $[u',v']$ be isomorphic Bruhat intervals in $S_n$. If $[u',v']$ is simple, then $\tR_{u,v}(q)=\tR_{u',v'}(q)$. Furthermore, for every strong hypercube decomposition $I$ of $[u,v]$ we have $\tR_{u,v}(q)=\tH_{u,v,I}(q)$.
\end{corollary}
\begin{proof}
    Let $\phi:[u,v]\to [u',v']$ be an isomorphism of posets and $I$ a strong hypercube decomposition of $[u,v]$. If $I$ is a \emph{proper} subset of $[u,v]$ then we know that $\tR_{u,x}(q)=\tR_{u',\phi(x)}(q)$ for all $x\in I$ by induction on $\ell(u,v)$. (Here we are using the fact that lower subintervals of the simple interval $[u',v']$ are simple.) Hence by following \Cref{rem:isomH}, we can show $\tH_{u,v,I}(q)=\tH_{u',v',\phi(I)}$. Furthermore, by \Cref{thm:simpleRH}, we have $\tR_{u',v'}(q) = \tH_{u',v',\phi(I)}(q)$. Combining these equalities, we see that $\tH_{u,v,I} = \tR_{u',v'}(q)$ whenever $I\subsetneq [u,v]$. Now \Cref{thm:existsequality} implies that there is some strong hypercube decomposition $I_{\mathrm{st}}\subsetneq [u,v]$ such that $\tH_{u,v,I_{\mathrm{st}}}(q) = \tR_{u,v}(q)$. As a result, $\tR_{u,v}(q) = \tR_{u',v'}(q)$ and the claim follows by induction.
\end{proof}

\begin{proof}[Proof of \Cref{thm:existsequality}]
    We will show that the \emph{standard hypercube decomposition} introduced in \cite{Blundell} is a strong hypercube decomposition admitting a reflection order which satisfies property (E). Hence by \Cref{thm:RHinequality}, the $\tH$-polynomial of the hypercube decomposition coincides with the $\tR$-polynomial of the interval.

    Let $u<v$ be elements of $S_n$. Then there is some value $1\leq i \leq n$ such that $u^{-1}(i) \neq v^{-1}(i)$. Let $d$ denote the minimal such value of $i$. The standard hypercube decomposition is
    \[ I_{\mathrm{st}} \coloneqq \{ x \in [u,v] \mid u^{-1}(d) = x^{-1}(d)  \}. \]
    It follows from the characterization of Bruhat order by projection to parabolic quotients \cite{Deodhar} that $y^{-1}(i)=u^{-1}(i)=v^{-1}(i)$ for $i<d$ for all $y \in [u,v]$. This also makes it clear that $[u,v] \cong [u',v']$ where $u'$ is obtained from $u$ by deleting $1,2,\ldots,d-1$ from the one-line notation and standardizing the alphabet, and similarly for $v'$. This isomorphism sends $I_{\mathrm{st}}$ to the standard hypercube decomposition $I'_{\mathrm{st}}$ of $[u',v']$; furthermore, $d' \coloneqq \min \{ i \mid u'^{-1}(i) \neq v'^{-1}(i)\}=1$. Then \cite[Theorem 5.1]{Blundell} 
    implies that $I'_{\mathrm{st}}$ is a diamond-closed lower interval of $[u',v']$, so the same is true of $I_{\mathrm{st}} \subset [u,v]$. 
    
    Similarly, \cite[Theorems 5.1 and 5.4]{Blundell} imply that for each $x\in I_{\mathrm{st}}'$, the set $\cY_x$ can be identified with $\{i\mid u^{-1}(1)< i \leq n \text{ and } (i,1)\cdot x \leq v\}$ so that a subset $Y=\{i_1<i_2<\ldots<i_k\}\subset \cY_x$ is in $\cA_x$ if and only if $x(i_1)>x(i_2)>\ldots>x(i_k)$. Furthermore, with this identification, the function $\theta_x:\cA_x\to [u',v']$ given by
    \[ \theta_x : \{ i_1 < i_2 < \ldots < i_k\} \mapsto x\cdot (u^{-1}(1),i_1,i_2,\ldots,i_k) \]
    is the unique function satisfying (HC1-3). Hence if we can show that $\theta_x$ satisfies (HC4) for each $x$, then we will have shown that $I_{\mathrm{st}}'$ is a strong hypercube decomposition of $[u',v']$ and thus that $I_{\mathrm{st}}$ is a strong hypercube decomposition of $[u,v]$. Consider a subgraph of $\Gamma(u',v')$ of the form
    \[\begin{tikzcd}[cramped,sep=small] & w & \\ \theta_x(Y)\ar[ur]& &\theta_x(Y')\ar[ul] \\ &\theta_x(Y\cap Y')\ar[ul]\ar[ur]& \end{tikzcd}.\] 
    Write $Y\cap Y' = \{ i_1 < \ldots < i_k\}$ and let $j$ (resp. $j'$) denote the additional element in $Y$ (resp. $Y'$). Without loss of generality, $j<j'$. Then there exist $i,i'\in \{x^{-1}(1),i_1,\ldots,i_k\}$ such that the diagram above is equal to
    \[\begin{tikzcd}[sep=small] & w & \\ \theta_x(Y\cap Y')\cdot(j,i)\ar[ur]& &\theta_x(Y\cap Y')\cdot(j',i')\ar[ul] \\ & x\cdot(u^{-1}(1),i_1,\ldots,i_k)\ar[ul]\ar[ur]& \end{tikzcd}.\] 
    If $(j,i)$ and $(j',i')$ commute, then $i<i'$ and so 
$\{i_1,\ldots,i,j,\ldots,i',j',\ldots,i_k\}$ is an antichain and 
    \begin{align*}w&= \theta_x(Y\cap Y')\cdot (j,i) \cdot (j',i')\\
    &= x\cdot(u^{-1}(1),i_1,\ldots,i,j,\ldots,i',j',\ldots,i_k)\\
    &= \theta_x(Y\cup Y').
    \end{align*}
    If they do not commute, then we must have $i=i'$. There are two cases. In the first we have $x(j)>x(j')$, from which it follows that $\{i_1,\ldots, i,j,j',\ldots,i_k\}$ is an antichain and
    \begin{align*}w&=\theta_x(Y\cap Y')\cdot (j',i) \cdot (j,i)\\
    &= x\cdot (u^{-1}(1),\ldots,i,j,j',\ldots,i_k)\\
    &= \theta_x(Y\cup Y').
    \end{align*}
    Note that $\theta_x(Y)\cdot (j',i)$ is below $\theta_x(Y)$ in Bruhat order and thus could not be equal to $w$. The second case is that $x(j)<x(j')$. In fact this case is impossible: the two options for $w$ are $\theta_x(Y)\cdot (j,i)$ and $\theta_x(Y')\cdot (j',i)$. These are both less than $\theta_x(Y')$, so this case cannot happen. 

    Finally, we need to show that there is a reflection order which satisfies property (E) relative to $I_{\mathrm{st}}$. Indeed, we can take the order guaranteed by \Cref{lem:constructreflorder} using $t_1=(d+1,d+2),t_2=(d+2,d+3),\ldots,t_i=(n-1,n)$ and $t_{i+1}=(1,2),\ldots,t_{k-1}=(d-1,d),t_k=(d,d+1)$. This order satisfies property (E) due to the following fact: if $x,y\in [u,v]$ are such that $x\in I_{\mathrm{st}}$ and $x\xrightarrow{t} y$, then $y\in I_{\mathrm{st}}$ if and only if $\alpha_t$ is in the span of $\alpha_{t_1},\ldots,\alpha_{t_i}$.  
\end{proof}

\begin{remark}
    Remarkably, a statement equivalent to $\tR_{u,v}(q) = \tH_{u,v,I_{\mathrm{st}}}(q)$ for the standard hypercube decomposition was shown by Brenti in \cite[Corollary 3.9]{Brenti1997}, before the introduction of hypercube decompositions in \cite{Blundell}. That we independently converged on this result after studying the work of \cite{Blundell, davies} is further testament to the usefulness of the artificial intelligence tools used there
    as a method for mathematical discovery.
\end{remark}

\subsection{Extension to other types}
In this paper we focused on the Coxeter group $S_n$. Many of the theorems can be stated for an arbitrary Coxeter group, and one could ask what results remain true in those cases. When $I$ is a strong hypercube decomposition of a simple interval $[u,v]$, our proof that $\tR_{u,v}(q) = \tH_{u,v,I}(q)$ works identically in types $D$ and $E$, and with some extra effort (and a modification to condition (HC4)) should work for an arbitrary Coxeter group. However, this is not enough to prove combinatorial invariance for those intervals. The problem is that $\tH_{u,v,I}(q)$ is defined in terms of $\tR$-polynomials for other intervals (which are strictly smaller if $I\subsetneq [u,v]$). In order to prove inductively that isomorphic intervals have equal $\tR$-polynomials, we need to know that there exists an $I\subsetneq [u,v]$ giving a (strong) hypercube decomposition of $[u,v]$, and similarly for all lower subintervals of $[u,v]$. As an obstruction to this, in \cite[Remark 3.9]{Blundell} the authors give an example of an interval in the group $H_3$ which has no proper hypercube decompositions. 

We do not know of any interesting families of intervals outside of type $A$ which are known to have such an abundance of hypercube decompositions, but the recent work of \cite{Gurevich-Wang} may suggest more general decompositions that always occur.

\section{The conjecture of Brenti--Marietti}

Brenti and Marietti have independently made a conjecture, which we now describe, somewhat similar to \Cref{conj:inequality}. The present work appeared on the arXiv 
before their paper \cite{Brenti-Marietti} was available to us, although an early version of their work had circulated privately. 

We first recall the original definition of hypercube decomposition from \cite{Blundell}.

\begin{definition}[\cite{Blundell}, Section~3.4]
\label{def:bbdvw-hcd}
A diamond-closed order ideal $I=[u,z] \subset [u,v]$ is a \emph{hypercube decomposition} if for each $x \in I$ and each $Y \in \cA_x$, there is a unique $|Y|$-hypercube subgraph of $\Gamma(u,v)$ with source $x$ and containing the edges $x\to y$ for $y \in Y$.
\end{definition}

A hypercube decomposition $[u,z] \subset [u,v]$ is \emph{meet} \cite[Def.~5.4]{Brenti-Marietti} if, for all $w \in [u,v]$, the set $[u,z] \cap [u,w]$ has a unique maximal element. Brenti and Marietti prefer \emph{upper} hypercube decompositions wherein $I$ is an order filter of $[u,v]$ and which are poset-dual to our conventions. In this section we continue the use of lower hypercube decompositions for ease of comparison with the rest of the paper. 

\begin{definition}[\cite{Brenti-Marietti}, Definition~3.1]
\label{def:shortcut}
For $x \leq y \in W$, let $d(x,y)$ denote the fewest number of edges in a directed path from $x$ to $y$ in $\Gamma$. The set of \emph{shortcuts} with respect to an element $z \in [u,v]$ is
\[
W_{[u,v]}^z \coloneqq \{p \in [u,z] \mid \gamma \cap [u,z] = \{p\} \text{ for all paths } \gamma \text{ from $p$ to $v$ of length $d(p,v)$}\}.
\]
\end{definition}

\begin{conjecture}[\cite{Brenti-Marietti}, Conjecture~6.1]
\label{conj:shortcut}
Let $[u,z]$ be a meet hypercube decomposition of $[u,v] \subset S_n$. Then 
\[
\tR_{u,v}(q)=\sum_{p \in W_{[u,v]}^z} q^{d(p,v)}\tR_{u,p}(q).
\]
\end{conjecture}

When comparing \Cref{conj:shortcut} with the definition (\ref{eq:H-definition}) of $\tH_{u,v,I}(q)$ and \Cref{conj:inequality}, it is natural to hope that 
\[
W_{[u,v]}^z = \{x \in [u,z] \mid \theta_x(Y)=v \text{ for some $Y \in \cA_x$}\}
\]
for $[u,z]$ a meet hypercube decomposition of $[u,v] \subset S_n$. It is not known whether this is true. See \cite{our-bbdvw-paper} for another condition on hypercube decompositions, the \emph{numerical criterion}, which may be equivalent to the meet hypotheses, and for which \Cref{conj:inequality} is conjectured to hold with equality. 

\section*{Acknowledgements}
We are grateful to Francesco Brenti, Mario Marietti, Lauren Williams, and Geordie Williamson for their helpful comments. GTB was supported by a National Science Foundation grant DMS-1854512.

\bibliographystyle{plain}
\bibliography{arxiv-v5}
\end{document}